\documentclass[12pt, oneside]{amsart}
\calclayout
\usepackage{amsmath}
\usepackage{amssymb}
\usepackage{amsthm} 
\usepackage{blindtext}
\usepackage{graphicx}
\usepackage{pgf,tikz}
\usepackage{mathrsfs}
\usepackage{comment}
\usetikzlibrary{arrows}
\usepackage{MnSymbol}
\usepackage{MnSymbol,wasysym}
\usepackage{array}
\usepackage{mathtools} 
\usepackage{graphicx}

\usepackage[english]{babel}
\usepackage[utf8]{inputenc}
\usepackage{amssymb}
\usepackage{pb-diagram}
\usepackage{mathtools}
\usepackage{array}
\usepackage{tabularx}
\usepackage{pgfkeys}
\usepackage{enumerate}
\usepackage{csquotes}
\usepackage[
    backend=biber,
giveninits=true,
style=alphabetic,
maxbibnames=99
]{biblatex}

\usepackage{amsthm}
\usepackage{amsmath}

\usepackage{hyperref}
\hypersetup{
	colorlinks=true,
	linkcolor=blue,
	filecolor=blue,      
	urlcolor=blue,
	citecolor=blue,
}
\usepackage{cleveref}

\usepackage{todonotes}

\numberwithin{equation}{section}

\theoremstyle{plain}
\newtheorem{lemma}[equation]{Lemma}
\newtheorem{theorem}[equation]{Theorem}

\newtheorem{corollary}[equation]{Corollary}
\newtheorem{proposition}[equation]{Proposition}
\theoremstyle{definition}
\newtheorem{definition}[equation]{Definition}
\newtheorem{remark}[equation]{Remark}

\newtheorem{example}[equation]{Example}
\newtheorem{problem}[equation]{Problem}

\DeclareMathOperator{\pd}{pd}
\DeclareMathOperator{\rk}{rk}
\DeclareMathOperator{\codim}{codim}
\DeclareMathOperator{\Ext}{Ext}
\DeclareMathOperator{\Hom}{Hom}
\DeclareMathOperator{\Der}{Der}

\newcommand{\Q}{\mathbb{Q}}
\newcommand{\A}{\mathcal{A}}
\newcommand{\B}{\mathcal{B}}
\newcommand{\C}{\mathcal{C}}

\newcommand{\V}{\mathcal{V}}

\newcommand{\KK}{\mathbb{K}}

\newcommand{\Sym}{\mathfrak{S}}

\newcommand\ddxi[1]{\partial_{x_{#1}}}

\mdseries
\title{Projective dimension of weakly chordal \\ graphic arrangements}
\author{Takuro Abe}
\address{Department of Mathematics, Rikkyo University, 3-34-1 Nishi-Ikebukuro, Toshima-ku, 1718501 Tokyo, Japan}
\email{abetaku@rikkyo.ac.jp}

\author{Lukas Kühne}
\address{Fakult\"at f\"ur Mathematik,
	Universit\"at Bielefeld, D-33501 Bielefeld, Germany}
\email{lkuehne@math.uni-bielefeld.de}

\author{Paul Mücksch}
\address{Fakult\"at f\"ur Mathematik, Ruhr-Universit\"at Bochum,
D-44780 Bochum, Germany}
\email{paul.muecksch+uni@gmail.com}

\author{Leonie Mühlherr}
\address{Fakult\"at f\"ur Mathematik,
	Universit\"at Bielefeld, D-33501 Bielefeld, Germany}
\email{lmuehlherr@math.uni-bielefeld.de}
\begin{document} 
	
	\begin{abstract}
		A graphic arrangement is a subarrangement of the braid arrangement whose set of hyperplanes is determined by an undirected graph.
		A classical result due to Stanley, Edelman and Reiner states that a graphic arrangement is free if and only if
		the corresponding graph is chordal, i.e., the graph has no chordless cycle with four or more vertices.
		In this article we extend this result by proving that the module of logarithmic derivations of a
		graphic arrangement has projective dimension at most one if and only if the corresponding graph is weakly chordal, 
		i.e., the graph and its complement have no chordless cycle with five or more vertices.
	\end{abstract} 
	%	\maketitle
	%\tableofcontents

	%%%%%%%%%%%%%%%%%%%%%%%%%%%%%%%%%%%%%%%%%%%%%%%%%%%%%%%%%%%%%%%%%%%%%%%%%%%%%%%%%%%%%%%
	
	%	\keywords{}
	%\subjclass[2010]{Primary: Secondary: }
	
	\maketitle
	
	%%%%%%%%%%%%%%%%%%%%%%%%%%%%%%%%%%%%%%%%%%%%%%%%%%%%%%%%%%%%%%%%%%%%%%%%%%%%%%%%%%%%%%%
	
	\section{Introduction}
	\label{sec:Introduction}
	
	The principal algebraic invariant associated to a hyperplane arrangement $\A$ is its module
	of logarithmic vectors fields or derivation module $D(\A)$.
	Such modules provide an interesting class of finitely generated graded modules over the coordinate ring of the ambient space
	of the arrangement.
	The chief problem is to relate the algebraic structure of $D(\A)$ to the combinatorial structure of $\A$,
	i.e.,\ whether it is free or more generally to determine its projective dimension or even graded Betti numbers.
	In general, this is notoriously difficult and still wide open, at its center is Terao's famous conjecture which states that
	over a fixed field of definition, the freeness of $D(\A)$ is completely determined by combinatorial data.
	Conversely, one might ask which combinatorial properties of $\A$ are determined by the algebraic structure of $D(\A)$.
	
	It is natural to approach these very intricate questions by restricting attention to certain distinguished classes of arrangements.
	
	A prominent and much studied class are the \emph{graphic arrangements}, around which our present article revolves.
	They are defined as follows.
	\begin{definition} 
		Let $V \cong \Q^\ell$ be an $\ell$-dimensional $\Q$-vector space.
		Let $x_1,..,x_\ell$ be a basis for the dual space $V^*$. 
		Given an undirected graph $G = (\V,E)$ with $\V=\{1,\dots,\ell\}$, define an arrangement $\A(G)$ by \[\A(G) \coloneqq \{\ker(x_i-x_j) \vert \{i,j\}\in E\}.\]
	\end{definition} 
	
	Our aim is to study the module $D(\A(G))$ of a graphic arrangement $\A(G)$.
	In fact, regarding the freeness of $D(\A(G))$, a nice complete answer is given by the following theorem,
	due to work by Stanley \autocite{Sta72}, and Edelman and Reiner \autocite{Edelman}.
	\begin{theorem}[{\autocite[Thm.~3.3]{Edelman}}]
		\label{thm:ChordalFree}
		The module $D(\A(G))$ is free if and only if the graph $G$ is chordal,
		i.e.,\ $G$ does not contain a chordless cycle with four or more vertices. 
	\end{theorem}
	
	A recent refined result was established in \autocite{tran2022matfree} by Tran and Tsujie, who showed that the subclass of so-called strongly chordal graphs in the class of chordal graphs corresponds to the subclass of MAT-free arrangements, cf.\ \autocite{ABCHT16_FreeIdealWeyl}, \autocite{CunMue19_MATfree}.
	
	In this note, we will investigate the natural question raised by Kung and Schenck in~\cite{KungSchenck} of whether it is possible to give a characterization of graphs $G$,
	similar to \Cref{thm:ChordalFree},
	for which the projective dimension of $D(\A(G))$ is bounded by a certain positive value.
	To this end, we consider the more general notion of \emph{weakly chordal} graphs introduced by Hayward~\autocite{Hayward1}:
    
	\begin{definition}
		A graph $G$ is weakly chordal if $G$ and its complement graph $G^C$ do not contain a chordless cycle with five or more vertices.
	\end{definition}
	It was subsequently discovered that many algorithmic questions that are intractable for arbitrary graphs become efficiently solvable within the class of weakly chordal graphs \autocite{Hayward2}.
 
	The main result of this paper is the following: 
	\begin{theorem}\label{thm:main}
		The projective dimension of $D(\A(G))$ is at most $1$ if and only if the graph $G$ is weakly chordal.
		Moreover, the projective dimension is exactly $1$ if $G$ is weakly chordal but not chordal.
	\end{theorem}
	
	Along the way towards the preceding theorem, we will prove the following key result, yielding the more difficult implication
	of \Cref{thm:main}.
	
	\begin{theorem}
		\label{thm:AntiholePD2}
		For $\ell\geq 6$, the projective dimension of $D(\A(C_\ell^C))$ is equal to $2$, 
		where $C_\ell^C$ is the complement of the cycle-graph with $\ell$ vertices, also called the ($\ell$-)antihole.
	\end{theorem}
	
	Moreover, we prove a refined result. Namely, in Theorem \ref{thm:antihole} we provide an explicit minimal free resolution
	of $D(\A(C_\ell^C))$.
	
	\bigskip
	
	The article is organized as follows.
	In \Cref{sec:PreliminariesGraphs}, we introduce some notation for graphs and preliminary results needed later on.
	\Cref{sec:Arrangements} is concerned with further notation and helpful results for hyperplane arrangements and their derivation modules.
	Moreover, in Subsection \ref{ssec:PolyB} we record a new tool from the very recent work of the first author \autocite{Abe23_BSequence} 
	which allows us to control the projective dimension of the derivation module along the deletion of hyperplanes under certain assumptions.
	Then, in \Cref{sec:WeaklyChordal} we prove one direction of our main Theorem \ref{thm:main}.
	The \Cref{sec:PDAntihole} then yields, step by step, the other direction of Theorem \ref{thm:main}.
	In particular, along the way, we derive a minimal free resolution of the derivation module of an antihole graphic arrangement.
	To conclude, in the final \Cref{sec:OpenProblems} we comment on open ends and record some questions 
    raised by our investigations.
	
	\section*{Acknowledgments}
	TA 
	is partially supported by JSPS KAKENHI Grant Number JP21H00975.
	LK and LM are supported by the Deutsche Forschungsgemeinschaft (DFG, German Research Foundation) -- SFB-TRR 358/1 2023 -- 491392403.
    PM is supported by a JSPS Postdoctoral Fellowship for Research in Japan.
	%%%%%%%%%%%%%%%%%%%%%%%%%%%%%%%%%%%%%%%%%%%%%%%%%%%%%%%%%%%%%%%%%%%%%%%%%%%%%%%%%%%%%%%
	%%%%%%%%%%%%%%%%%%%%%%%%%%%%%%%%%%%%%%%%%%%%%%%%%%%%%%%%%%%%%%%%%%%%%%%%%%%%%%%%%%%%%%
	
	\section{Preliminaries -- Graph Theory} 
	\label{sec:PreliminariesGraphs}
	In this section, we define objects of interest to us while studying graphic arrangements, notably specific graph classes and their attributes.
	The exposition is mostly based on \autocite{Diestel}.
	We only consider simple, undirected graphs:
	\begin{definition}
		\begin{enumerate} 
			\item[(i)] A simple graph $G$ on a set $\V$ is a tuple $(\V,E)$ with $E \subseteq \binom{\V}{2}$
			the set of (undirected) edges connecting the vertices in $\V$. 
			\item[(ii)] The graph $G^C = \left(\V, \binom{\V}{2}\backslash E\right)$ is called the \emph{complement graph} of $G$. 
			\item[(iii)] A graph $G' = (\V', E')$ with $\V' \subseteq \V, E' \subseteq E$ is called a \emph{subgraph} of $G$. 
			If $E'$ is the set of all edges between vertices in $\V'$, i.e.\ $E'= \binom{\V'}{2}\cap E$, the graph $G'$ is an \emph{induced subgraph} of $G$. 
		\end{enumerate} 
	\end{definition} 
	If the subset relation is proper, $G'$ is called a \emph{proper subgraph} of $G$.
	
	Besides restricting the graph to a set of vertices, there are two basic operations we can perform on graphs, as described in \autocite{OrlikTerao}: 
	\begin{definition}\label{adddel}
		Let $G = (\V,E)$ be a graph and $e = \{i,j\}\in E$. 
		\begin{enumerate}
			\item The graph $G' 
   = (\V, E\backslash\{e\})$ is obtained from $G$ through deletion of $e$.
			\item The graph $G'' = (\V'', E'')$ with $V''$ the vertex set obtained by identifying
			$i$ and $j$ and $E'' = \{\{\bar{p}, \bar{q}\}~\vert \{p,q\}\in E'\}$ is obtained by contraction of $G$ with respect to~$e$.  
		\end{enumerate} 
	\end{definition} 
	We will define graph classes based on certain path or cycle properties: 
	\begin{definition}
		\begin{enumerate}[(i)]
			\item For $k\ge 2$, a path of length $k$ is the graph $P_k = (\V,E)$ of the form \[\V = \{v_0,\dots,v_k\}~,~E = \{\{v_0,v_1\}, \{v_1,v_2\},\dots,\{v_{k-1},v_k\}\}\] where all $v_i$ are distinct.  
			\item If $P_k = (\V,E)$ is a path, and $k \geq 3$, then the graph $C_k = (\V, E\cup \{v_{k-1},v_0\})$ is called a ($k$-)cycle.
		\end{enumerate}
	\end{definition}
	An edge which joins two vertices of a cycle (path), but is not itself an edge of the cycle (path) is a \emph{chord} of that cycle (path). 
	An induced cycle (path) of a graph $G$ is an induced subgraph of $G$, that is a cycle (path).
	For $k\geq 6$, we call $C_k^C$ the $k$-\emph{antihole}.
	\begin{definition} 
		A graph is called \emph{chordal} (or \emph{triangulated}) if each of its cycles of length at least 4 has a chord, i.e.\ if it contains no induced cycles of length greater than 3. 
	\end{definition} 
	The main objects of interest in this article are graphs that satisfy a weaker condition than chordality and were introduced by Hayward in \autocite{Hayward1}: 
	\begin{definition}
		A graph is called \emph{weakly chordal} (or \emph{weakly triangulated}) if it contains no induced $k$-cycle with $k \geq 5$ and no complement of such a cycle as an induced subgraph. 
	\end{definition} 
	
	It is clear that chordality implies weak chordality and that weak chordality is closed under taking the complement. 
	Additionally, it is apparent that if $G$ is weakly chordal, so is every induced subgraph of $G$ and in \autocite{Hayward2}, it was proved that weak chordality is closed under contraction.
	
	\bigskip
	A more inductive approach is given by the following generation method, introduced by Hayward: 
	\begin{theorem}\emph{(\autocite{Hayward3}, Theorem 4)}\label{thm:wc}  
		A graph is weakly chordal if and only if it can be generated in the following manner: 
		\begin{enumerate} 
			\item Start with a graph $G_0$ with no edges. 
			\item Repeatedly add an edge $e_j$ to $G_{j-1}$ to create the graph $G_j$, such that $e_j$ is not the middle edge of any induced $P_3$ of $G_j$.  
		\end{enumerate} 
	\end{theorem} 
	With these tools, we can now prove the following: 
	\begin{lemma}\label{lem:wc}
		For a weakly chordal graph $G = (\V,E)$, there exists a sequence of edges $e_1,..,e_k\notin E$, such that 
		\begin{enumerate}
			\item $G_i = (\V, E \cup \{e_1,\dots,e_i\})$ is weakly chordal for $i = 1,\dots,k-1$,
			\item the edge $e_i$ is not part of an induced cycle $C_4$ in $G_i$ for $i = 1,\dots,k$ and
			\item $G_k$ is chordal.
		\end{enumerate}
	\end{lemma} 
	\begin{proof}
		Say the complement $G^C$ has $m$ edges.
		If $G$ is weakly chordal, so is $G^C$. Using \Cref{thm:wc} 
        this means in turn that there exists an edge ordering $e_m,\dots,e_1$ of the edges in $E_{G^C}$, such that $(\V, \{e_m,\dots,e_i\})$ is weakly chordal for all $i=m,\dots,1$ and $(\V, \{e_m,\dots,e_1\})=G^C$.
		
		Define the sequence of graphs $G_i \coloneqq (\V, E \cup \{e_1,\dots,e_i\})$ for $i=1,\dots,m$.
		As $G_i^C=(\V, \{e_m,\dots,e_{i-1}\})$ these graphs are by construction all weakly chordal.
		Since the sequence ends with the complete graph, which is chordal, the chordality condition is met at some point in the sequence.
		Moreover, the middle edge in an induced path $P_4$ becomes an edge on an induced cycle $C_4$ in the complement graph.
		Thus the condition of~\Cref{thm:wc} on avoiding the middle edges of an induced $P_4$ translates to avoiding the edges of an induced cycle $C_4$ as claimed.
	\end{proof}

	%%%%%%%%%%%%%%%%%%%%%%%%%%%%%%%%%%%%%%%%%%%%%%%%%%%%%%%%%%%%%%%%%%%%%%%%%%%%%%%%%%%%%%%
	%%%%%%%%%%%%%%%%%%%%%%%%%%%%%%%%%%%%%%%%%%%%%%%%%%%%%%%%%%%%%%%%%%%%%%%%%%%%%%%%%%%%%%
	
	\section{Preliminaries -- Hyperplane Arrangements} 
	\label{sec:Arrangements}
	
	In this section, we recall some fundamental notions form the theory of hyperplane arrangements.
	The standard reference is Orlik and Terao's book \autocite{OrlikTerao}.
	
	\begin{definition} 
		Let $\KK$ be a field and let $V \cong \KK^\ell$  be a $\KK$-vector space of dimension $\ell$. 
		A hyperplane $H$ in $V$ is a linear subspace of dimension $\ell-1$. 
		A hyperplane arrangement $\A = (\A, V)$  is a finite set of hyperplanes in $V$.
	\end{definition} 
	
	Let $V^*$ be the dual space of $V$ and $S = S(V^*)$ be the symmetric algebra of $V^*$. 
	Identify $S$ with the polynomial algebra $S = \KK[x_1,\dots,x_\ell]$. 
	\begin{definition} 
		Let $\A$ be a hyperplane arrangement. Each hyperplane $H \in \A$ is the kernel of a polynomial $\alpha_H$ of degree 1 defined up to a constant. 
		The product \[Q(\A) \coloneqq \prod_{H \in \A} \alpha_H\] is called a defining polynomial of $\A$. 
	\end{definition}  
	
	Define the \emph{rank} of $\A$ as $\rk(\A) := \codim_V(\cap_{H \in \A}H)$.
	If $\mathcal{B} \subseteq \A$ is a subset, then $(\mathcal{B}, V)$ is called a subarrangement. 
	The \emph{intersection lattice} $L(\A)$ of the arrangement is the set of all non-empty intersections of elements of $\A$ (including $V$ as the intersection over the empty set), with partial order by reverse inclusion. 
    For $X\in L(\A)$ define the \emph{localization} at $X$ as the subarrangement $\A_X$ of $\A$ by
	\[\A_X \coloneqq \{H \in \A~\vert~X \subseteq H\}\] as well as the \emph{restriction} $(\A^X, X)$ as an arrangement in $X$ by
	\[\A^X \coloneqq \{X \cap H~\vert~H \in \A\backslash\A_X~\text{and}~X\cap H \not= \emptyset\}.\]
	Define \[L_k(\A) \coloneqq \{X \in L(\A)~\vert~\text{codim}_V (X) = k\} \] and $L_{\geq k}(\A), L_{\leq k}(\A)$ analogously. 
	
	\begin{definition} 
		Let $\A$ be a non-empty arrangement and let $H_0 \in \A$. Let $\mathcal{A'} = \A\backslash\{H_0\}$ and let $\mathcal{A''} = \A^{H_0}$. We call $(\A, \mathcal{A'}, \mathcal{A''})$ a triple of arrangements with distinguished hyperplane $H_0$. 
	\end{definition} 
	
	We can associate a special module to the hyperplane arrangement $\A$: 
	\begin{definition} 
		A $\mathbb{K}$-linear map $\theta: S \rightarrow S$ is a derivation if for $f,g \in S$:
		\[\theta(f\cdot g) = f\cdot \theta(g)+g\cdot \theta(f).\]
		Let $\Der_{\mathbb{K}}(S)$ be the $S$-module of derivations of $S$.
		This is a free $S$-module with basis the usual partial derivatives $\ddxi{1},\dots,\ddxi{\ell}$.
		
		Define an $S$-submodule of $\Der_{\mathbb{K}}(S)$, called the module of $\A$-derivations, by
		\[D(\A) \coloneqq \{\theta \in \Der_{\mathbb{K}}(S) \vert \theta(Q) \in QS\}.\]
		The arrangement $\A$ is called free if $D(\A)$ is a free $S$-module. 
	\end{definition} 
	
	The class of arrangements we are interested in are graphic arrangements: 
	\begin{definition} 
		Given a graph $G = (\V,E)$ with $\V=\{1,\dots,\ell\}$, define an arrangement $\A(G)$ by \[\A(G) \coloneqq \{\ker(x_i-x_j) \vert \{i,j\}\in E\}.\]
	\end{definition} 
	
	\begin{remark}
		Note that for a graphic arrangement $\A(G)$, localizations exactly correspond
		to disconnected unions of induced subgraphs of $G$. More precisely,
		for $X \in L(\A(G))$ we have
		$\A(G)_X = \{\ker(x_i-x_j) \mid \{i,j\} \in E'\}$ for some $E' \subseteq E$
		if and only if there is a subgraph $G'$ of $G$ with edges $E'$
		such that each connected component of $G'$ is an induced subgraph of $G$.
	\end{remark}
	
	\medskip
	
	For given derivations $\theta_1, \ldots, \theta_\ell \in \Der(S)$ we define the the \emph{coefficient matrix}
	\[
	M(\theta_1, \ldots, \theta_\ell) := \left(\theta_j(x_i)\right)_{1\leq i,j \leq \ell},
	\]
	i.e.,\ the matrix of coefficients with respect to the standard basis
	$\ddxi{1},\ldots,\ddxi{\ell}$
	of $\Der(S)$.
	We recall Saito's useful criterion for the freeness of $D(\A)$, cf.~\autocite[Thm.~4.19]{OrlikTerao}.
	
	\begin{theorem}
		\label{thm:SaitosCriterion}
		For $\theta_1, \dots, \theta_\ell \in D(\A)$, the following are equivalent:
		\begin{enumerate}[(1)]
			\item 
			$\det(M(\theta_1, \ldots, \theta_\ell)) \in  \KK^\times\, Q(\A)$,
			
			\item
			$\theta_1, \dots, \theta_\ell$ is a basis of $D(\A)$.
		\end{enumerate}
	\end{theorem}

	%%%%%%%%%%%%%%%%%%%%%%%%%%%%%%%%%%%%%%%%%%%%%%%%%%%%%%%%%%%%%%%%%%%%%%%%%%%%%%%%%%%%%%
	
	\subsection{Projective dimension}
	\label{ssec:PD}
	In this manuscript, we want to take a look at the non-free case of graphic arrangements and find a characterization for their different projective dimensions.
	For a comprehensive account of all the required homological and commutative algebra notions we refer to \autocite{Weibel}
	respectively \autocite{Eis95_CommAlg}.
	
	\begin{definition} 
		An $S$-module $P$ is called projective if it satisfies the following universal lifting property: 
		given $S$-modules $L,N$, a surjection $g: L \rightarrow N$ and a map $\gamma: P \rightarrow N$, 
		there exists a map $\beta: P \rightarrow N$, such that $\gamma= g \circ \beta$.
		
		A projective resolution of a module $M$ is a complex $P_\bullet$ with a map $\epsilon: P_0 \rightarrow M$, such that the augmented complex 
		\[\dots \rightarrow P_2 \rightarrow P_1 \rightarrow P_0 \xrightarrow{\epsilon} M \rightarrow 0\] is exact and $P_i$ is projective for all $i\in \mathbb{N}$.
	\end{definition} 
	\begin{lemma} 
		Every $S$-module $M$ has a projective resolution. 
	\end{lemma} 
	
	With this in mind, we can define the notion of projective dimension: 
	\begin{definition}
		Let $M$ be an $S$-module. Its \emph{projective dimension} $\pd(M)$ is the minimum integer $n$ (if it exists), 
		such that there is a resolution of $M$ by projective $S$-modules 
		\[0 \rightarrow P_n \rightarrow \dots\rightarrow P_1 \rightarrow P_0 \rightarrow M \rightarrow 0\] 
	\end{definition} 
	
	\begin{remark}
		\label{rem:PDExt}
		The projectivity of the $S$-module $P$ is equivalent to the exactness of the functor $\Hom_S(P,-)$.
		Hence, considering its derived functors $\Ext_S^i(M,-)$, we have the following characterization of the projective dimension,
		cf.\ \autocite[pd Lemma 4.1.6]{Weibel}:
		\[
		\pd(M) \leq p \iff \Ext_S^i(M,N) = 0 \text{ for all } i>p \text{ and all }S\text{-modules } N.
		\]
	\end{remark}
	
	The projective dimension of an arrangement is the projective dimension of its derivation module
	and we simply write $\pd(\A) := \pd(D(\A))$. 
	Note that $D(\A)$ is a finitely generated reflexive module over the polynomial ring $S$;
	as such we have $\pd(\A) \leq \rk(\A)-2$ and (as a consequence of the graded version of Nakayama's Lemma) $D(\A)$ is projective if and only it is free,
	cf.\ \autocite[Thm.~19.2]{Eis95_CommAlg}.
	Thus, by \Cref{thm:ChordalFree}, a chordal graph produces an arrangement of projective dimension $0$.
	
	The following result is due to Terao, cf.\ \autocite[Lem.~2.1]{Yuz91_LatticeCohom}.
	\begin{proposition}
		\label{prop:LocPD}
		Let $X \in L(\A)$. Then $\pd(\A_X) \leq \pd(\A)$.
	\end{proposition}
	
	An arrangement $\A$ is \emph{generic}, if $|\A| > \rk(\A)$
	and for all $X \in L(\A) \setminus \{\cap_{H \in \A}H\}$ we have $|\A_X| = \codim_V(X)$.
	The next result, due to Rose and Terao \autocite{RoseTerao1991_FreeResGeneric},
	identifies generic arrangements as those with maximal projective dimension.
	\begin{theorem}
		\label{thm:pdGeneric}
		Let $\A$ be a generic arrangement. Then $\pd(\A) = \rk(\A)-2$.
	\end{theorem}
	
	Important for our present investigations are the following examples of generic arrangements.
	\begin{example}
		\label{ex:CycleGeneric}
		Let $C_\ell$ be the cycle graph with $\ell$ vertices.
		Then, for $\ell\geq 3$, the graphic arrangement $\A(C_\ell)$ is generic.
		In particular, we have $\pd(\A(C_\ell)) = \rk(\A(C_\ell))-2 = \ell-3$.
	\end{example}
	
	Since arrangements of induced subgraphs correspond to localizations, from
	\Cref{ex:CycleGeneric} and \Cref{prop:LocPD} we obtain the following,
	first observed by Kung and Schenck \autocite[Cor.~2.4]{KungSchenck}.
	\begin{corollary}
		\label{coro:InducedCyclePDBound} 
		If $G$ contains an induced cycle of length  $m$, then $\emph{pd}(\A(G)) \geq m-3$.
	\end{corollary}
	
	In \autocite{KungSchenck}, Kung and Schenck introduced a graph they called the ``triangular prism'' to serve as an example for a graphic arrangement $\A(G)$ whose projective dimension is strictly greater than $k-3$, $k$ the length of the longest chordless cycle in $G$.
	Note that the graph they describe is the $6$-antihole, see~\Cref{fig:prism}.
	It does not have any cycle of length 5 or more, yet $\pd(\A(G))=2$ and it is not weakly chordal. 
	It is the smallest possible example (in terms of the number of vertices) that has this property. 
	\begin{figure}
		\includegraphics[width=0.6\textwidth]{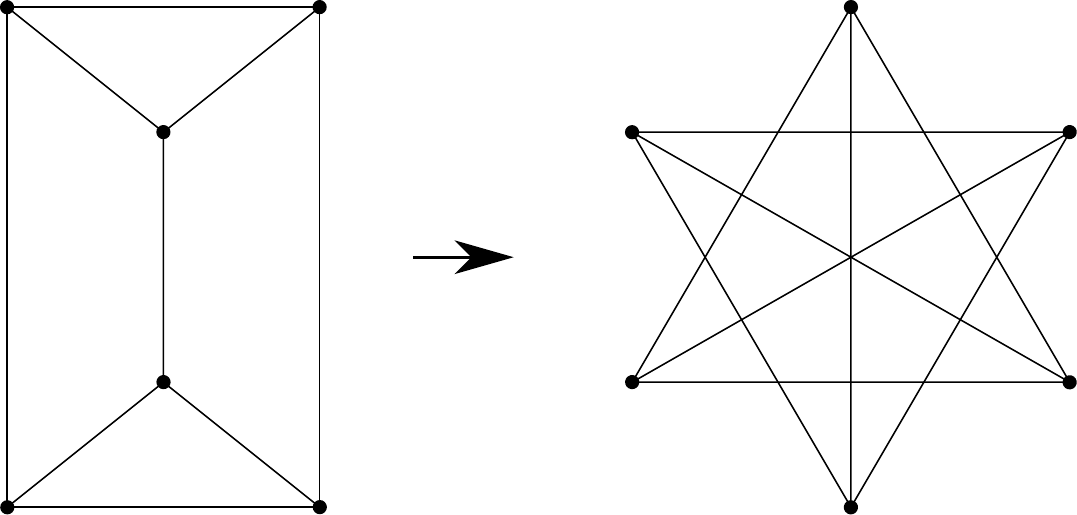} 
		\caption{The triangular prism of~\autocite{KungSchenck} on the left is the same as ${C_6}^C$ on the right.} 
		\label{fig:prism}
	\end{figure}

	%%%%%%%%%%%%%%%%%%%%%%%%%%%%%%%%%%%%%%%%%%%%%%%%%%%%%%%%%%%%%%%%%%%%%%%%%%%%%%%%%%%%%%%
	
	\subsection{Terao's polynomial \texorpdfstring{$B$}{B}}
	\label{ssec:PolyB}
	Let $\A$ be an arbitrary arrangement and $H_0$ a distinguished hyperplane.
	Let $(\A,\A',\A'')$ be the corresponding triple.
	Choose a map $\nu:\A'' \to \A'$ such that $\nu(X)\cap H_0=X$ for all $X\in \A''$.
	
	Terao defined the following polynomial
	\[
	B(\A',H_0)=\frac{Q(\A)}{\alpha_{H_0}\prod_{X\in\A''}\alpha_{\nu(X)}}.
	\]
	The main properties of this polynomial can be summarized as follows:
	\begin{proposition}\autocite[Lem. 4.39 and Prop. 4.41]{OrlikTerao}
		\begin{enumerate}
			\item $\deg B(\A',H_0) = |\A'|-|\A''|$.
			\item The ideal $(\alpha_{H_0},B(\A',H_0))$ is independent of the choice of $\nu$.
			\item The polynomial $\theta(\alpha_{H_0})$ is contained in the ideal $(\alpha_{H_0},B(\A',H_0))$ for all $\theta\in D(\A')$.
		\end{enumerate}
		\label{prop:B}
	\end{proposition}
	
	In the following, we fix a hyperplane $H_0$ and simply write $B = B(\A',H_0)$ for Terao's polynomial.
	
	By \Cref{prop:B}, we have an exact sequence:
	
	\begin{equation}\label{eq:seq}
		0\rightarrow D(\A) \hookrightarrow{} D(\A') \xrightarrow{\bar{\partial'}} \bar{S}\cdot \bar{B},
	\end{equation}
	where $\bar{S}=S/\alpha_{H_0}$ and $\overline{\partial'}(\theta)=\overline{\theta(\alpha_{H_0})}$.
	
	The following new result regarding this sequence will be important in our subsequent proofs.
	It is a special case of ``surjectivity theorems'' for sequences of local functors
	recently obtained by the first author in \autocite{Abe23_BSequence}.
	
	\begin{theorem}\label{thm:surjective}
		Assume that $\pd(\A_X) < \codim_V(X)-2$ for all $X \in L_{\geq 2}(\A^{H_0})$.
		Then the map $\overline{\partial'}$ in the sequence~\eqref{eq:seq} is surjective.
		Hence, in this case, the sequence~\eqref{eq:seq} is also right exact. 
	\end{theorem}
	\begin{proof}
		This immediately follows from \autocite[Thm.~3.2, Thm.~3.3]{Abe23_BSequence}.
	\end{proof}
	
	We record the following consequences of the preceding theorem.
	\begin{corollary}
		\label{cor:SeqBRightExact}
		Assume that $\A_X$ is free for all $X\in L_2(\A^{H_0})$ and $\pd(\A) \leq 1$. Then the sequence~\eqref{eq:seq} is also right exact.
	\end{corollary}
	\begin{proof}
		This follows immediately from \Cref{thm:surjective} and \Cref{prop:LocPD}.	
	\end{proof}
	
	\begin{lemma}
		\label{lem:SeqBRightExactPd1}
		Assume that $\A_X$ is free for all $X\in L_2(\A^{H_0})$ and $\pd(\A) \leq 1$. Then we also have $\pd(\A') \leq 1$.
	\end{lemma}
	\begin{proof}
		By \Cref{cor:SeqBRightExact}, the $B$-sequence is right-exact
		and by assumption, $\pd(\A) = 0$, that is $\Ext_S^i(D(\A),N) = 0$ for all $i>1$
		and all $S$-modules $N$.
		The principal ideal of $\bar{S}$ generated by $\bar{B}$ is free as an $\bar{S}$-module. 
		So, by the graded version of \autocite[Cor.~4.3.14]{Weibel}, 
		the module $\bar{S}\bar{B}$ has projective dimension $1$, i.e.\ $\Ext_S^i(\bar{S}\bar{B},N) = 0$ for all $i>1$ by \Cref{rem:PDExt}.
		It then follows from the long exact $\Ext$-sequence, that for the middle term in the $B$-sequence,
		we have $\Ext_S^i(D(\A'),N) = 0$ for all $i>1$ which is equivalent to $\pd(\A') = \pd(D(\A')) \leq 1$
		by \Cref{rem:PDExt}.
	\end{proof}
	
	%%%%%%%%%%%%%%%%%%%%%%%%%%%%%%%%%%%%%%%%%%%%%%%%%%%%%%%%%%%%%%%%%%%%%%%%%%%%%%%%%%%%%%%
	%%%%%%%%%%%%%%%%%%%%%%%%%%%%%%%%%%%%%%%%%%%%%%%%%%%%%%%%%%%%%%%%%%%%%%%%%%%%%%%%%%%%%%
	
	\section{Weakly chordal graphic arrangements}
	\label{sec:WeaklyChordal}
	
	The goal of this section is to show that a graphic arrangement of a weakly chordal graph has projective dimension at most $1$,
	which gives one direction of our main \Cref{thm:main}.
	
	\begin{theorem}
		\label{thm:pdWeaklyChordal1}
		Let $G=(\V,E)$ be a weakly chordal graph.
		Then $\pd (\A(G)) \le 1$.
	\end{theorem}
	\begin{proof}
		Firstly, \Cref{lem:wc} implies that there exists a sequence of edges $e_1,\dots,e_k$ such that $G_i = (\V, E \cup \{e_1,..,e_i\})$ is weakly chordal,
		the edge $e_i$ is not the middle edge of any induced $P_4$ in $G_i$ for $i = 1,..,k$, and $G_k$ is chordal.
		
		We prove that $\pd (\A(G_i)) \le 1$ for all $i = 1,..,k$ by a descending induction.
		As $G_k$ is chordal, the arrangement $\A(G_k)$ is free and hence $\pd (\A(G_k))=0$ by \Cref{thm:ChordalFree}.
		So assume that $\pd (\A(G_j)) \le 1$ for some $1<j\le k$. We will now argue that this implies $\pd (\A(G_{j-1})) \le 1$ which finishes the proof.
		
		Let $H_0$ be the hyperplane corresponding to the edge $e_j$ in the arrangement $\A(G_j)$.
		We aim to apply  \Cref{lem:SeqBRightExactPd1} to $\A(G_j)$ and $\A(G_{j-1})$.
		To check the assumption of this result, we consider $X\in L_2(\A(G_j)^{H_0})$ and need to show that the arrangement $\A(G_j)_X$ is free.
		
		Assume the contrary, i.e., that $\A(G_j)_X$ is not free.
		By definition of $X$, the arrangement $\A(G_j)_X$ is a graphic arrangement on an induced subgraph of $G_j$ on four vertices containing the edge $e_j$.
		The assumption that this arrangement is not free implies that this induced subgraph is not chordal.
		As this subgraph only contains four vertices it must be the cycle $C_4$.
		This however contradicts condition (2) in~\Cref{lem:wc} which states that the edge $e_j$ cannot be an edge of an induced cycle $C_4$ in the graph~$G_j$.
		Therefore, the arrangement $\A(G_j)_X$ is free for all $X\in L_2(\A(G_j)^{H_0})$.
		
		Moreover, by the induction hypothesis, we have $\pd(\A(G_j)) \leq 1$.
		Thus, by Lemma \ref{lem:SeqBRightExactPd1}, we also have $\pd(\A(G_{j-1})) \le 1$ as desired.
	\end{proof}

	Let us record the following result which immediately follows from the previous theorem and Theorem \ref{thm:ChordalFree}.
	\begin{corollary}
		\label{coro:WChordNotChordPD1}
		Let $G$ be a weakly chordal but not chordal graph. Then $\pd(\A(G)) = 1$.
	\end{corollary}

	%%%%%%%%%%%%%%%%%%%%%%%%%%%%%%%%%%%%%%%%%%%%%%%%%%%%%%%%%%%%%%%%%%%%%%%%%%%%%%%%%%%%%%%
	%%%%%%%%%%%%%%%%%%%%%%%%%%%%%%%%%%%%%%%%%%%%%%%%%%%%%%%%%%%%%%%%%%%%%%%%%%%%%%%%%%%%%%
	
	\section{Graphic arrangements of antiholes}
	\label{sec:PDAntihole}
	
	The main result of this section yields the other direction of implications in \Cref{thm:main}.
	Recall that the graph $C_\ell^C$ is the complement graph of a cycle with $\ell$ vertices which is called the $\ell$-\emph{antihole}.
	\begin{theorem}[Theorem \ref{thm:AntiholePD2}]
		\label{thm:antholes}
		For all $\ell\ge 6$ it holds that
		\[
		\pd(\A(C_\ell^C))= 2.
		\]
	\end{theorem}
	
	Before we delve into the arguments, leading step by step to the above principal theorem of this section,
	let us first explain how this concludes the proof of \Cref{thm:main}.
	\begin{proof}[Proof of \Cref{thm:main}, using \Cref{thm:antholes}]
		By \Cref{thm:pdWeaklyChordal1}, we have $\pd(\A(G)) \leq 1$ for a weakly chordal graph~$G$
		and $\pd(\A(G)) = 1$ if $G$ is not chordal by \Cref{coro:WChordNotChordPD1}.
		
		Conversely, assume that $G$ is a graph such that $\pd(\A(G)) = 1$.
		In particular, by \Cref{thm:ChordalFree}, the graph $G$ is not chordal.
		Suppose $G$ is also not weakly chordal. 
		Then, by definition, there is either an $m\geq 5$ such that $C_m$ is an induced subgraph or
		there is an $\ell \geq 6$ such that $C_\ell^C$ is an induced subgraph of $G$.
		In the first case, by \Cref{coro:InducedCyclePDBound}, we have $\pd(\A(G)) \geq \ell-3 \geq 2$;
		in the second case, by \Cref{prop:LocPD} and \Cref{thm:antholes}, we also have
		$\pd(\A(G)) \geq 2$. Both cases contradict our assumption.
		Hence, $G$ is weakly chordal.	
	\end{proof}

	To prove Theorem \ref{thm:antholes}, let us first introduce some notation for
	special derivations we will consider in this section. 
	Let $G$ be a graph with vertex set $$
	\V=
	[\ell]:=\{1,2,\ldots,\ell\}.
	$$ 
	Write
	$H_{ij}:=\ker(x_i-x_j)$ for the hyperplane corresponding to the edge $\{i,j\}$ and let 
	$$
	\A_{\ell-1}:=\{H_{ij} \mid 1 \le i < j \le \ell\}
	$$
	be the graphic arrangement of the complete graph (or equivalently, the Weyl arrangement of type $A_{\ell-1}$, also called the braid arrangement)
	in $\mathbb{Q}^\ell$.
	We set
	$$
	\theta_i:=\sum_{j=1}^\ell x_j^i \ddxi{j}\ (i \ge 0)
	$$
	and define
	$$
	\varphi_i:=\prod_{j \in [\ell] \setminus \{i-1,i,i+1\}} (x_i-x_j)\ddxi{i}
	$$
	for $i \neq 1, \ell$.
	Also define 
	$$
	\varphi_1:=\prod_{i=3}^{\ell-1} (x_1-x_i)\ddxi{1}
	$$
	and 
	$$
	\varphi_\ell:=\prod_{i=2}^{\ell-2} (x_\ell-x_i)\ddxi{\ell}.
	$$
	
	In this section we always consider indices and vertices in $[\ell]$ in a cyclic way,
	i.e.,\ we identify $i+\ell$ with $i$ etc.
	
	There is the following fundamental result due to K.\ Saito.
	
	\begin{theorem}[\autocite{Saito80}]
		$\A_{\ell-1}$ is free with basis 
		$$
		\theta_0,\ldots,\theta_{\ell-1}.$$
		\label{thm:saito}
	\end{theorem}
	
	With this, we can show the following.
	
	\begin{lemma}\label{lemma:freedeletion}
		Let 
		$$
		\mathcal{B}_{i,j}:=\A_{\ell-1} \setminus \{H_{s,s+1} \mid 
		i \le s \le j\}.
		$$
		If $j=i+2$, then $\B_{i,i+2}$ is free with basis 
		$$
		\theta_0,\ldots,\theta_{\ell-3}, \varphi_{i+1},\varphi_{i+2}.
		$$
	\end{lemma}
	
	\begin{proof}
		We use Saito's criterion (\Cref{thm:SaitosCriterion}).
		Apparently $\theta_0,\dots,\theta_{\ell-3},\varphi_{i+1},\varphi_{i+2} \in D(\B_{i,i+2})$.
		Considering the coefficient matrix $M(\theta_0,\ldots,\theta_{\ell-3},\varphi_{i+1},\varphi_{i+2})$,
		to compute its determinant, we can expand it along the last two columns,
		yielding a smaller Vandermonde determinant $\prod_{1\leq s < t \leq \ell, i,j \notin \{i+1,i+2\}}(x_s-x_t)$
		multiplied with the only entries in the last two columns $\prod_{j \in [\ell] \setminus \{i,i+1,i+2\}} (x_{i+1}-x_j)$
		and $\prod_{j \in [\ell] \setminus \{i+1,i+2,i+3\}} (x_{i+2}-x_j)$.
		But the product of these three terms is exactly the defining polynomial $Q(\B_{i,i+2})$ which yields the freeness
		of $\B_{i,i+2}$ by Saito's criterion.
	\end{proof}

	\begin{proposition}\label{prop:deletingSPOG}
		If $i+2\le j \neq i-1$, then $D(\B_{i,j})$ is generated by 
		$$
		\theta_0,\ldots,\theta_{\ell-3}, \varphi_{i+1},\varphi_{i+2},\ldots,\varphi_{j}.
		$$
	\end{proposition}
	
	\begin{proof}
		Firstly, the defining polynomials of the $\B_{i,j}$ together with the derivations $\varphi_{i+1},\varphi_{i+2},\ldots,\varphi_{j}$ for a fixed $|i-j| = s$ 
		are contained in one orbit under the action of the symmetric group $\Sym_\ell$ on $S=\mathbb{Q}[x_1,\ldots,x_\ell]$ respectively on subsets of $\Der(S)$.
		Hence, without loss, we may assume that $i=1$.
		
		We argue by induction on $j$. By Lemma \ref{lemma:freedeletion}, the statement is true for $j=3$. 
		Assume that $D(\B_{1,j})$ is generated by 
		$$
		\theta_0,\ldots,\theta_{\ell-3}, \varphi_{2},\varphi_{3},\ldots,\varphi_{j}.
		$$
		We will show that, after deleting $H_{j+1,j+2}$, an additional generator $\varphi_{j+1}$ is necessary, i.e.\ 
		$D(\B_{1,j+1})$ is generated by 
		$$
		\theta_0,\ldots,\theta_{\ell-3}, \varphi_{2},\varphi_{3},\ldots,\varphi_{j}, \varphi_{j+1}.
		$$
		Apparently, we have
		\begin{align*}
			|\B_{1,j+1}| &= |\A_{\ell-1}|-(j+1) = \frac{\ell(\ell-1)}{2}-(j+1),\\
			|\B_{1,j}^{H_{j+1,j+2}}| &= |\A_{\ell-2}|-(j-1) = \frac{(\ell-1)(\ell-2)}{2}-(j-1).\\ 
		\end{align*}
		Thus 
		$$
		\deg B_{j+1}=|\B_{1,j+1}|-|\B_{1,j}^{H_{j+1,j+2}}|
		=\ell-3,
		$$
		where $B_{j+1} = B(\B_{1,j+1},H_{j+1,j+2})$ is Terao's polynomial from Subsection \ref{ssec:PolyB}.
		By definition, it is clear that 
		$\varphi_{j+1} \in D(\B_{1,j+1}) \setminus D(\B_{1,j})$. 
		Consequently, by \Cref{prop:B}, we have $\varphi_{j+1}(x_{j+1}-x_{j+2}) = g(x_{j+1}-x_{j+2}) + cB_{j+1}$
		and for any $\theta \in D(\B_{1,j+1})$ we also have $\theta(x_{j+1}-x_{j+2}) = g'(x_{j+1}-x_{j+2}) + fB_{j+1}$,
		for certain $f,g,g' \in S$ and $c \in \mathbb{Q}^\times$.
		Hence, $$\theta - \frac{f}{c}\varphi_{j+1} \in D(\B_{1,j})=\langle \theta_0,\ldots,\theta_{\ell-3}, \varphi_2,\ldots,\varphi_{j}\rangle_S$$
		by the induction hypothesis, which completes the proof.  
	\end{proof}
	
	We thus see, that if we delete $H_{12},\ldots,H_{\ell-1,\ell}$ from $\A_{\ell-1}$,
	we can determine generators for $D(\B_{1,\ell-1})$, namely 
	$$
	D(\B_{1,\ell-1})=\langle
	\theta_0,\ldots,\theta_{\ell-3},\varphi_2,\ldots,\varphi_
	{\ell-1}\rangle.
	$$ 
	However, the same argument as in Proposition \ref{prop:deletingSPOG} does not work well
	for our target arrangement $\A(C_\ell^C) = \B_{1,\ell} = \B_{1,\ell-1} \setminus \{H_{\ell,1}\}$, since 
	\begin{eqnarray*}
		|\B_{1,\ell}|&=&\displaystyle \frac{\ell(\ell-1)}{2}-\ell,\\
		|\B_{1,\ell-1}^{H_{\ell,1}}|&=&\displaystyle \frac{(\ell-1)(\ell-2)}{2}-(\ell-3).
	\end{eqnarray*}
	So $|\B_{1,\ell}|-|\B_{1,\ell-1}^{H_{\ell,1}}|=\ell-4=\deg B_\ell<\deg(\varphi_\ell) = \deg(\varphi_1) = \ell-3$, 
	where $B_\ell = B(\B_{1,\ell},H_{\ell,1})$ is Terao's polynomial $B$.
	To obtain generators for $\B_{1,\ell}$ we need to modify the argument utilizing the polynomial $B$. 
	For that purpose, we introduce the following new refined version of \Cref{prop:B}.
	
	\begin{theorem}
		Let $\A$ be an arrangement, $H_1, H_2 \not \in \A$ be distinct hyperplanes and let 
		$\A_i:=\A \cup \{H_i\}$. Assume that $H_1 = \ker(\alpha)$, $H_2 = \ker(\beta)$ and let 
		$B_i$ be the polynomial $B$ with respect to $(\A,H_i)$. 
		Assume that 
		$\ker (\alpha+\beta) \in \A$, let $b$ be the greatest common divisor of the reduction of $B_1$ and $B_2$ modulo $(\alpha,\beta)$ 
		and let $b_2b\equiv B_2$ modulo $(\alpha,\beta)$. 
		Then for $\theta \in D(\A)$ we have:
		$$
		\theta(\alpha) 
		\in (\alpha, \beta B_1,b_2B_1).
		$$ 
		\label{thm:idealB}
	\end{theorem}
	\begin{proof} 
		Let 
		$$
		\theta(\alpha)=f\alpha+FB_1
		$$
		for some $f,F \in S$.
		Note that 
		$$
		\theta(\alpha)=\theta(\alpha+\beta)-\theta(\beta).
		$$
		So 
		$$
		\theta(\alpha)=g(\alpha+\beta)+h\beta+aB_2
		$$
		for some $g,h,a \in S$, and thus, we have 
		$$
		f\alpha+FB_1=g(\alpha+\beta)+h\beta+aB_2.
		$$
		Reducing the equation modulo $\alpha$, we obtain 
		$$
		FB_1 \equiv g\beta + h\beta + aB_2 \mod (\alpha).
		$$
		Reducing once more modulo $\beta$, we get 
		$$
		FB_1 \equiv aB_2 \mod (\alpha,\beta).
		$$
		Let $B_i \equiv bb_i \mod (\alpha,\beta)$. Then 
		$$
		F=F_1 b_2+F_2 \beta+F_3 \alpha
		$$
		for some $F_1,F_2,F_3 \in S$.
		Hence
		$$
		\theta(\alpha) \in 
		(\alpha, \beta B_1,b_2B_1),
		$$
		which completes the proof. 
	\end{proof}
	
	We can apply Theorem \ref{thm:idealB} to $\B_{1,\ell-1}$ and $\B:= \B_{1,\ell} = \A_{\ell-1} \setminus \{H_{1,2},\ldots,H_{\ell-1,\ell},H_{\ell,1}\}$. 
	Namely, we can show the following:
	
	\begin{theorem}
		$$
		D(\B)=\langle\theta_0,\ldots, \theta_{\ell-3},\varphi_1,\ldots,\varphi_\ell \rangle_S.
		$$
		\label{thm:generator}
	\end{theorem}
	
	\begin{proof}
		Let 
		$\C_1:=\B \cup \{H_{12}\}$ and $\C_2:=\B \cup \{H_{23}\}$.
		Set
		$$
		B_1=\prod_{j=4}^{\ell-1} (x_1-x_j)
		$$
		for the polynomial $B$ for the pair $(\B,H_{12})$ and 
		$$
		B_2=\prod_{j=5}^{\ell} (x_2-x_j)
		$$
		for the polynomial $B$ of $(\B,H_{23})$.
		Note that $H_{i,i+1} \not \in \B$ $(i=1,2)$ and 
		$(x_1-x_2)+
		(x_2-x_3) =x_1-x_3$, whose kernel is in $\B$. 
		Moreover, after reduction modulo $x_1=x_2=x_3$, we have 
		$$
		B_1 \equiv \prod_{j=4}^{\ell-1} (x_1-x_j),\ 
		B_2 \equiv \prod_{j=5}^{\ell} (x_1-x_j),
		$$
		and their common divisor is $\prod_{j=5}^{\ell-1} (x_1-x_j)$. 
		Then, Theorem \ref{thm:idealB} yields for $\theta \in D(\B)$ 
		\begin{eqnarray*}
			\theta(x_1-x_2) &\in& (x_1-x_2,(x_2-x_3)B_2,(x_2-x_\ell)B_1)=(x_1-x_2,(x_1-x_3)B_2,(x_1-x_\ell)B_1)\\
			&=&(x_1-x_2,\varphi_1(x_1-x_2),\varphi_2(x_1-x_2) ),
		\end{eqnarray*}
		where we used the fact that 
		$$
		(x_2-x_3)B_2 \equiv (x_1-x_3)B_2,\ 
		(x_2-x_\ell)B_1 \equiv (x_1-x_\ell)B_1 \mod(x_1-x_2). 
		$$
		Thus 
		\begin{align*}
			D(\B) &= D(\C_1)+S\varphi_1+S\varphi_2
			= \langle \theta_0,\ldots,\theta_{\ell-3},\varphi_1,\ldots,\varphi_\ell \rangle_S,
		\end{align*}
		by \Cref{prop:deletingSPOG}.
	\end{proof}

	Note that $\psi_{i}:=(x_{i-1}-x_i)\varphi_i-(x_{i+1}-x_{i+2})\varphi_{i+1}  \in 
	D(\A_{\ell-1}) = \langle \theta_0,\ldots,\theta_{\ell-1}\rangle_S$ for $i=1,\ldots,\ell$,
	since
	\[
	\psi_i(x_i-x_{i+1}) = -\prod_{j \in [\ell]\setminus \{i,i+1\}}(x_i-x_j) + \prod_{j \in [\ell]\setminus \{i,i+1\}}(x_{i+1}-x_j) \equiv 0 \mod(x_i-x_{i+1}).
	\]
	Thus, there are $f_{ij}$ such that
	\begin{align}
		\label{eq:relPsiTheta}
		\psi_i-\sum_{j=0}^{\ell-3} f_{ij} \theta_j =-\theta_{\ell-2} \ (i=1,2,\ldots,\ell).
	\end{align}
	So we have relations 
	$$
	\psi_i-\sum_{j=0}^{\ell-3} f_{ij} \theta_j=
	\psi_s-\sum_{j=0}^{\ell-3} f_{sj} \theta_s,
	$$
	and they are generated by 
	\begin{equation}
		\psi_1-\sum_{j=0}^{\ell-3} f_{1j} \theta_j =
		\psi_{i}-\sum_{j=0}^{\ell-3} f_{ij} \theta_j
		\label{eq}
	\end{equation}
	for $i=2,\ldots,\ell$. 
	We now prove that they indeed generate all the relations among the generators of $D(\B)$.
	
	\begin{theorem}\label{thm:rel}
        All relations among the set of generators $\theta_0,\ldots,\theta_{\ell-3},\varphi_1,\ldots,\varphi_\ell$
        are generated by the ones given in Equations \eqref{eq}.
	\end{theorem}
	
	\begin{proof}
		Let 
		\[
		\eta:\quad \sum_{i=0}^{\ell-3} a_i \theta_i+\sum_{i=1}^\ell b_i \varphi_i=0
		\] 
		be a relation. 
		Since $\theta_i \in D(\A_{\ell-1})$ ($0\leq i \leq \ell-3$), we see that $(\sum_{i=1}^\ell b_i \varphi_i)(x_1-x_2)$ is divisible by $x_1-x_2$, i.e., 
		$$
		b_1\prod_{i=3}^{\ell-1} (x_1-x_i)-
		b_2\prod_{i=4}^{\ell} (x_2-x_i)
		$$
		is divisible by $x_1-x_2$. So 
		$$
		b_1(x_2-x_3)\equiv b_2(x_1-x_\ell) \mod (x_1-x_2).
		$$
		Hence, there are polynomials $g_{12},h_1,h_2 \in S$ such that
		\begin{align*}
			b_1 &= (x_1-x_\ell) g_{12}+(x_1-x_2)h_1,\\
			b_2 &= (x_2-x_3) g_{12}+(x_1-x_2)h_2.
		\end{align*}  
		Apply the same argument to 
		$(\sum_{i=1}^\ell b_i \varphi_i)(x_2-x_3)$ to obtain polynomials $g_{23} \in S$ such that 
		$$
		b_2=(x_2-x_3) g_{12}+(x_1-x_2)g_{23}+(x_2-x_3)(x_1-x_2)h_2.
		$$
		Continuing these processes, we know that 
		$$
		b_i=(x_{i}-x_{i+1}) g_{i-1,i}+(x_{i-1}-x_i)g_{i,i+1}+(x_{i-1}-x_i)(x_{i}-x_{i+1})h_i
		$$
		for some polynomials $g_{i-1,i}, g_{i,i+1}, h_i \in S,\ i=1,\ldots,\ell$. 
		
		Substituting this information into our relation $\eta$, we obtain
		\[
		\eta:\quad  \sum_{i=0}^{\ell-3} a_i \theta_i+\sum_{j=1}^\ell c_j\psi_j + (x_{j-1}-x_j)(x_{j}-x_{j+1})h_j \varphi_j=0,
		\]
		for some $a_i,c_j ,h_j \in S$.
		Applying our relations from Equations \eqref{eq},
		we get a relation of the form
		\[
		\sum_{i=0}^{\ell-3} t_i \theta_i+ t_{\ell-2}\psi_1 + \sum_{j=1}^\ell (x_{j-1}-x_j)(x_{j}-x_{j+1})h_j \varphi_j=0,
		\]
		Note that by Equations \eqref{eq:relPsiTheta} we have $\psi_1 = -\theta_{\ell-2} - \sum_{j=0}^{\ell-3} f_{1j} \theta_j$,
		and since $(x_{i-1}-x_i)(x_{i}-x_{i+1})\varphi_i \in D(\A_{\ell-1})$,
		we thus have $(x_{i-1}-x_i)(x_{i}-x_{i+1})\varphi_i = -\theta_{\ell-1} + \sum_{j=0}^{\ell-3} c_{ij} \theta_i-c_{\ell-2,i} \psi_1$
		for suitable $c_{ij} \in S$.
		Applying this last substitution to our relation, we now expressed $\eta$ as
		
		$$
		\sum_{i=0}^{\ell-3} t_i \theta_i+t_{\ell-2} \psi_1+t_{\ell-1} \theta_{\ell-1}=0.
		$$ 
		Since $\theta_0,\ldots,\theta_{\ell-3},\psi_1,\theta_{\ell-1}$ are linearly independent (in fact,
		by Equation \eqref{eq:relPsiTheta} they form a basis for $D(\A_{\ell-1})$),
		we have $t_i=0$ for all $i$.
		Consequently, all the relations among the above generators of $D(\B)$ are expressible using Equations (\ref{eq}).
	\end{proof}
	\medskip

	Recall that 
	$$
	(x_{i-1}-x_i)\varphi_i-(x_{i+1}-x_{i+2})\varphi_{i+1} 
	\in D(\A_{\ell-1} ),
	$$
	for $i=1,\ldots, \ell$
	and there are $f_{ij}$ such that 
	\begin{equation}
		\label{eq:ESP}
		(x_{i-1}-x_i)\varphi_i-(x_{i+1}-x_{i+2})\varphi_{i+1} +
		\sum_{j=0}^{\ell-3} f_{ij} \theta_j=-\theta_{\ell-2}.
	\end{equation}
	First let us show the following.
	
	\begin{lemma}
		The coefficients in relations (\ref{eq:ESP}) are given explicitly by 
		$$
		f_{ij}=(-1)^{\ell-2-j}e_{\ell-2-j}(x_1,\ldots,\hat{x}_i,\hat{x}_{i+1},\ldots,x_\ell),
		$$
		where $e_i(a_1,\ldots,a_{\ell-2})$ is the $i$-th basic symmetric polynomial in the variables 
		$a_1,\ldots,a_{\ell-2}$.
		\label{lemma:sp}
	\end{lemma}
	
	\begin{proof} 
		The straightforward computation is left to the reader.
	\end{proof}
	
	Now we are ready to prove the following, which immediately implies Theorem \ref{thm:AntiholePD2}.
	
	\begin{theorem}
		The module $D(\B)$ has the following minimal free resolution:
		\begin{equation}
			\label{eq:min}
			0 \rightarrow S[-\ell+1] \rightarrow 
			S[-\ell+2]^{\ell-1} \rightarrow 
			\oplus_{i=0}^{\ell-4} S[-i] 
			\oplus S[-\ell+3]^{\ell+1} \rightarrow D(\B) \rightarrow 0.
		\end{equation}
		In particular, $\pd(\B) =2$.
		\label{thm:antihole}
	\end{theorem}

	\begin{proof}
		First we prove the second syzygy part.
		Let 
		$$
		e_j^{i}:
		=e_{j}(x_1,\ldots,\hat{x}_i,\hat{x}_{i+1},\ldots,x_\ell).
		$$
		As we have shown in \Cref{lemma:sp}, the relations among the generators 
		$$
		\theta_0,\ldots,\theta_{\ell-3},\varphi_1,\ldots,\varphi_\ell
		$$
		are the following:
		\begin{eqnarray*}
			(x_{\ell}-x_1)\varphi_1&-&(x_{2}-x_{3})\varphi_{2}
			-\sum_{j=0}^{\ell-3} 
			(-1)^{\ell-2-j}e_{\ell-2-j}^1\theta_j\\
			=(x_{i-1}-x_i)\varphi_i&-&(x_{i+1}-x_{i+2})\varphi_{i+1}
			-\sum_{j=0}^{\ell-3} 
			(-1)^{\ell-2-j}e_{\ell-2-j}^i\theta_j
		\end{eqnarray*}
		for $i=2,\ldots,\ell$. 
		Let us denote these relations by $\psi_i$ for $i=2,\ldots,\ell$. 
		Since we are now concerned with the relations among those first syzygies, 
		from now on we consider all first syzygies as vectors with coordinates with respect to
		$\theta_0,\ldots,\theta_{\ell-3},\varphi_1,\ldots,\varphi_\ell$.
		
		Now let 
		$$
		\sum_{i=2}^\ell a_i \psi_i=0
		$$
		be a relation among our generators of the first syzygy module.
		Since the coefficients of $\varphi_i$ for $ 3 \le i \le \ell$ are 
		only $x_i-x_{i+1}$ in $\psi_{i-1}$ and $x_{i-1}-x_i$ in 
		$\psi_i$, we can deduce that 
		$$
		a_i=A(x_i-x_{i+1})
		$$
		for $i=2,\ldots,\ell$ and some constant polynomial $A \in S$. So the relations among $\psi_i$'s have to be of the form 
		$$
		A\sum_{i=2}^\ell (x_i-x_{i+1})\psi_i=0.
		$$
		Let us check whether other coefficients are zero or not.
		First, we consider the one of $\varphi_1$, that is 
		$$
		(x_\ell-x_1)(\sum_{i=2}^\ell (x_i-x_{i+1})+(x_1-x_2))=0.
		$$
		Second, for $\varphi_2$, we similarly have 
		$$
		(x_2-x_3)(-\sum_{i=2}^\ell (x_i-x_{i+1})-(x_1-x_2))=0.
		$$
		So it suffices to check the remaining part, i.e., the coefficient of each
		$\theta_{\ell-2-j}$, which is of the form  
		$$
		\sum_{i=2}^\ell (x_i-x_{i+1})(e_j^i-e_j^1)
		=\sum_{i=2}^\ell (x_i-x_{i+1})e_j^i+(x_1-x_2)e_j^1=
		\sum_{i=1}^\ell (x_i-x_{i+1})e_j^i=:C_j.
		$$
		We show this is zero by induction on $j \ge 1$. 
		For $j=1$ we have 
		\begin{eqnarray*}
			\ddxi{1}(C_1)&=&
			\sum_{i=3}^\ell x_i+\sum_{i=2}^{\ell-1}(x_i-x_{i+1})
			-\sum_{i=2}^{\ell-1}x_i\\
			&=&
			\sum_{i=3}^\ell x_i+(x_2-x_\ell)
			-\sum_{i=2}^{\ell-1}x_i=0.
		\end{eqnarray*}
		By an analogous computation we have $\ddxi{i}(C_1)=0$ for all $i$. Thus $C_1=0$. So assume that $C_s=0$ for $s \le j$ and let 
		us prove that $C_{j+1}=0$. 
		Compute 
		\begin{eqnarray*}
			\ddxi{1}(C_{j+1})&=&
			e_{j+1}^1+\sum_{i=2}^{\ell-1}(x_i-x_{i+1})e_{j}(x_2,\ldots,
			\hat{x}_i,\hat{x}_{i+1},\ldots,x_\ell)-e_{j+1}^\ell \\
			&=&
			\sum_{i=2}^{\ell-1}(x_i-x_{i+1})e_{j}(x_2,\ldots,
			\hat{x}_i,\hat{x}_{i+1},\ldots,x_\ell)+(x_\ell-x_2)
			e_{j}(x_3,\ldots,x_{\ell-1})\\
			&+&e_{j+1}(x_3,\ldots,x_{\ell-1})-
			e_{j+1}(x_3,\ldots,x_{\ell-1})=0
		\end{eqnarray*}
		by the induction hypothesis.
		
		In sum, we have established the exactness of our resolution.
		
		Now, recalling that all the coefficients in the 
		resolution (\ref{eq:min}) are of positive degree, we see that it is moreover a 
		minimal free resolution. This finishes the proof. 
	\end{proof}

	%%%%%%%%%%%%%%%%%%%%%%%%%%%%%%%%%%%%%%%%%%%%%%%%%%%%%%%%%%%%%%%%%%%%%%%%%%%%%%%%%%%%%%%
	%%%%%%%%%%%%%%%%%%%%%%%%%%%%%%%%%%%%%%%%%%%%%%%%%%%%%%%%%%%%%%%%%%%%%%%%%%%%%%%%%%%%%%
	
	\section{Open problems}\label{sec:OpenProblems}
	We conclude the article by mentioning a few open problems and further directions of research.
	
	For free graphic arrangements which correspond to chordal graphs by \Cref{thm:ChordalFree},
	the degrees of the generators in a basis of the derivation module have a nice description in terms of the combinatorics of the graph,
	namely vertex degrees along a vertex elimination orderings, cf.\ \autocite[Lem.~3.4]{Edelman}.
	Thus, the following problem arises from \Cref{thm:main}.
	\begin{problem}
		Determine the graded Betti numbers of $D(\A(G))$ for a weakly chordal graph $G$.
	\end{problem}
	
	A further natural question arising from our Theorem \ref{thm:main} would be if this generalizes to the remaining projective dimensions, i.e. if $\A(G)$ has projective dimension $\leq k$ if and only if $G$ and its complement graph do not contain a chordless cycle with $k+4$ or more vertices. This is however not the case, first note that in the case of projective dimension $0$, it suffices for the graph itself to have no chordless cycle of length $4$ or more and chordality is not closed under taking the complement (The complement of the $4$-cycle for instance, is chordal, where the $4$-cycle itself is not). 
	Moreover, since the arrangement of the $k$-cycle is generic of rank $k-1$, it has maximal projective dimension $k-3$ (see \Cref{ex:CycleGeneric})
	and by \Cref{thm:AntiholePD2} its complement has projective dimension $2$.
	Moreover, we found two counterexamples to the other direction of this conjecture in dimension 7; both graphs and their complements have no induced cycle of length more than 5, yet have projective dimension 3, see \Cref{fig:prism2} which was also found by Hashimoto in \cite{Hashimoto}.
	
	\begin{figure}
		\includegraphics[width=0.5\textwidth]{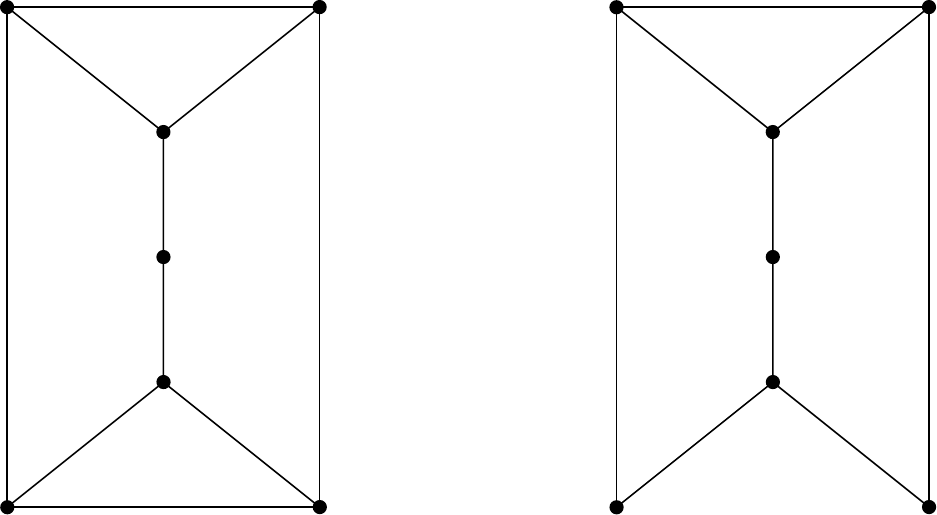} 
		\caption{Two graphs with 7 vertices, whose graphic arrangements have projective dimension 3.}
		\label{fig:prism2}
	\end{figure}
	
	Lastly, we would like to mention that the problem of understanding the projective dimension of the logarithmic $p$-forms $\Omega^p(\A)$ with poles along $\A$ (cf.\ \autocite[Def.\ 4.64]{OrlikTerao}) greatly differs from the one we discuss in this article.
	As the module $\Omega^1(\A)$ and the module $D(\A)$ are dual to each other, one of them is free if and only if the other is free.
	In the non-free scenario they behave differently however:
	For instance we have $\pd(\Omega^1(\A(C_\ell)))=1$ while we have $\pd(D(\A(C_\ell)))=\ell-3$ for $\ell\ge 4$.
	There are furthermore graphs $G$ with
	\[\pd(\Omega^1(\A(G)))>\pd(D(\A(G)));\]
	one such example is the complete graph on six vertices $K_6$ with three long diagonals removed.
	So it seems to be an interesting but intricate problem for further research to understand 
	for which graphic arrangements the projective dimension of the logarithmic $1$-forms is bounded by one.

	%%%%%%%%%%%%%%%%%%%%%%%%%%%%%%%%%%%%%%%%%%%%%%%%%%%%%%%%%%%%%%%%%%%%%%%%%%%%%%%%%%%%%%%
	%%%%%%%%%%%%%%%%%%%%%%%%%%%%%%%%%%%%%%%%%%%%%%%%%%%%%%%%%%%%%%%%%%%%%%%%%%%%%%%%%%%%%%
	
	\printbibliography
	
\end{document}